\documentclass[english]{article}
\usepackage[T1]{fontenc}
\usepackage{geometry}
\usepackage{cleveref}
\usepackage{color}
\usepackage{babel}
\usepackage{refstyle}
\usepackage{enumitem}
\usepackage{cite}
\usepackage{amsmath}
\usepackage{amsthm}
\usepackage{amssymb}
\usepackage{pict2e}
\usepackage[all]{xy}

\usepackage[unicode=true,pdfusetitle,
 bookmarks=true,bookmarksnumbered=true,bookmarksopen=false,
 breaklinks=false,pdfborder={0 0 0},pdfborderstyle={},backref=false,colorlinks=true]
 {hyperref}

\newlist{paraenum}{enumerate}{1}%
\setlist[paraenum]{label={(\arabic*)}}%
\newlist{lemenum}{enumerate}{1}%
\setlist[lemenum]{label={\emph{(\arabic*)}}}%

\definecolor{mygreen}{rgb}{0, 0.7, 0}
\definecolor{myblue}{rgb}{0, 0, 0.7}
\definecolor{mypurple}{RGB}{100, 0, 150}

\usepackage{color}
\usepackage[usenames,dvipsnames,svgnames,table]{xcolor}
\usepackage{hyperref}
\hypersetup{
     colorlinks   = true,
     citecolor    = mypurple,
     linkcolor    = myblue
}

\makeatletter

\AtBeginDocument{\providecommand\thmref[1]{\ref{thm:#1}}}
\AtBeginDocument{\providecommand\propref[1]{\ref{prop:#1}}}
\AtBeginDocument{\providecommand\lemref[1]{\ref{lem:#1}}}
\AtBeginDocument{\providecommand\defref[1]{\ref{def:#1}}}
\AtBeginDocument{}
\AtBeginDocument{}
\AtBeginDocument{\providecommand\corref[1]{\ref{cor:#1}}}
\AtBeginDocument{\providecommand\exref[1]{\ref{ex:#1}}}

  \theoremstyle{plain}
  \newtheorem{thm}{Theorem}[section]

  \theoremstyle{definition}

  \theoremstyle{definition}
  \newtheorem{war}[thm]{Warning}

  \theoremstyle{plain}
  \newtheorem{notation}[thm]{Notation}  

  \theoremstyle{plain}
  \newtheorem{lem}[thm]{\protect\lemmaname}

  \theoremstyle{definition}
  \newtheorem{defn}[thm]{\protect\definitionname}

  \theoremstyle{plain}
  \newtheorem{prop}[thm]{\protect\propositionname}

  \theoremstyle{plain}
  \newtheorem{cor}[thm]{\protect\corollaryname}

  \theoremstyle{definition}

  \newtheorem{ex}[thm]{\protect\examplename}

  \theoremstyle{remark}
  \newtheorem{rem}[thm]{\protect\remarkname}

  \theoremstyle{remark}
  \newtheorem*{rem*}{\protect\remarkname}

\@ifundefined{date}{}{\date{}}

\newref{prop}{name=Proposition~}
\newref{lem}{name=Lemma~}
\newref{thm}{name=Theorem~}
\newref{cor}{name=Corollary~}
\newref{rem}{name=Remark~}
\newref{def}{name=Definition~}
\newref{exa}{name=Example~}

\makeatother

  \providecommand{\corollaryname}{Corollary}
  \providecommand{\definitionname}{Definition}
  \providecommand{\examplename}{Example}
  \providecommand{\lemmaname}{Lemma}
  \providecommand{\propositionname}{Proposition}
  \providecommand{\remarkname}{Remark}

\begin{document}
\global\long\def\oto#1{\overset{#1}{\longrightarrow}}
\global\long\def\iso{\overset{\sim}{\longrightarrow}}
\global\long\def\into{\hookrightarrow}
\global\long\def\onto{\twoheadrightarrow}
\global\long\def\ss{\subseteq}
\global\long\def\adj{\leftrightarrows}
\global\long\def\bb#1{\mathbb{#1}}
\global\long\def\es{\varnothing}
\global\long\def\term{\text{pt}}
\global\long\def\cone{\triangleright}
\global\long\def\from{\leftarrow}

\global\long\def\colim{\operatorname*{colim}}
\global\long\def\lim{\operatorname*{lim}}
\global\long\def\map{\operatorname{Map}}
\global\long\def\End{\operatorname{End}}
\global\long\def\fun{\operatorname{Fun}}
\global\long\def\mul{\operatorname{Mul}}
\global\long\def\Id{\operatorname{Id}}
\global\long\def\red{\operatorname{\scriptsize{red}}}
\global\long\def\un{\operatorname{\scriptsize{un}}}
\global\long\def\final{\operatorname{\scriptsize{final}}}
\global\long\def\act{\operatorname{\scriptsize{act}}}
\global\long\def\grp{\operatorname{\scriptsize{grp}}}

\global\long\def\hocolim{\operatorname*{hocolim}}
\global\long\def\holim{\operatorname*{holim}}
\global\long\def\alg{\operatorname{Alg}}

\global\long\def\LP#1#2#3#4{\xymatrix{#1\ar[r]\ar[d]  &  #3\ar[d]\\
 #2\ar[r]\ar@{-->}[ru]  &  #4 
}
 }
\global\long\def\Square#1#2#3#4{\xymatrix{#1\ar[r]\ar[d]  &  #3\ar[d]\\
 #2\ar[r]  &  #4 
}
 }

\global\long\def\Rect#1#2#3#4#5#6{\xymatrix{#1\ar[r]\ar[d]  &  #3\ar[d]\ar[r]  &  #5\ar[d]\\
 #2\ar[r]  &  #4\ar[r]  &  #6 
}
 }

\global\long\def\catd{\mathbf{Cat}_{d}}
\global\long\def\cat{\mathbf{Cat}_{\infty}}
\global\long\def\catpt{\mathbf{Cat}_{\infty,*}}
\global\long\def\op{\mathbf{Op}_{\infty}}
\global\long\def\oppt{\mathbf{Op}_{\infty,*}}
\global\long\def\opred{\mathbf{Op}_{\infty}^{\red}}
\global\long\def\opun{\mathbf{Op}_{\infty}^{\un}}
\global\long\def\opunpt{\mathbf{Op}_{\infty,*}^{\un}}
\global\long\def\opredpt{\mathbf{Op}_{\infty,*}^{\red}}
\global\long\def\opd{\mathbf{Op}_{d}}
\global\long\def\pop{\mathbf{POp}_{\infty}}
\global\long\def\poppt{\mathbf{POp}_{\infty,*}}

\global\long\def\sseq{\mathbf{SSeq}}
\global\long\def\fin{\mathbf{Fin}}
\global\long\def\set{\mathbf{Set}}
\global\long\def\sset{\mathbf{sSet}}
\global\long\def\finpt{\mathbf{Fin}_{*}}
\global\long\def\uni{\mathbf{Uni}}
\global\long\def\com{\mathbf{Com}}
\global\long\def\ass{\mathbf{Ass}}
\global\long\def\triv{\mathbf{Triv}}

\title{On $d$-Categories and $d$-Operads}

\author{Tomer M. Schlank and Lior Yanovski}
\maketitle

\begin{abstract}
We extend the theory of $d$-categories, by providing an explicit description of the right mapping spaces of the $d$-homotopy category of an $\infty$-category. Using this description, we deduce an invariant $\infty$-categorical characterization of the $d$-homotopy category. We then proceed to develop an analogous theory of $d$-operads, which model $\infty$-operads with $(d-1)$-truncated multi-mapping spaces, and prove analogous results for them. 

\end{abstract}

\tableofcontents

\section{Introduction}

\paragraph{Overview \& Organization.}
The notion of a \emph{$d$-category} was introduced by Lurie in \cite[2.3.4]{HTT}, as a strict model for what we call an \emph{essentially $d$-category}; 
An $\infty$-category all of whose mapping spaces are homotopically $(d-1)$-truncated.
With any $\infty$-category $\mathcal{C}$, Lurie associates a $d$-category $h_d\mathcal{C}$, which we call the \emph{$d$-homotopy category} of $\mathcal{C}$. While this $d$-category is shown to be universal in the 1-categorical (simplicially enriched) sense among $d$-categories that $\mathcal{C}$ is mapped to \cite[2.3.4.12]{HTT}, the question of how does $h_d\mathcal{C}$ relate to $\mathcal{C}$ as an $\infty$-category, is left  unaddressed. The goal of this note is to fill this gap and to give an analogous treatment for operads. 

In section 2, we begin by showing that the right mapping spaces of $h_d\mathcal{C}$ are given, upto isomorphism, by applying $h_{d-1}$ to the right mapping spaces of $\mathcal{C}$ (\propref{h_d_mapping_spaces_cat}). 
This is the main technical result of this note, the proof of which goes through the comparison with the ``middle mapping spaces''. From this we deduce that $h_d\mathcal{C}$ is obtained from $\mathcal{C}$ by $(d-1)$-truncation of the mapping spaces. More precisely, we show that $h_d$ can be promoted to a functor of $\infty$-categories, which is left adjoint to the inclusion of the full subcategory spanned by essentially $d$-categories into $\cat$. And furthermore, that the unit map of this adjunction is essentially surjective and is given on mapping spaces by the $(d-1)$-truncation map (\thmref{d-homotopy_category}). 

In section 3, we develop a parallel theory for operads. We call an $\infty$-operad an \emph{essentially $d$-operad} if all of its multi-mapping spaces are $(d-1)$-truncated. We begin by defining a notion of a \emph{$d$-operad} (\defref{d_operad}) that relates to essentially $d$-operads in the same way that $d$-categories relate to essentially $d$-categories. We then define the \emph{$d$-homotopy operad} functor (\defref{d_homotopy_operad}), again by analogy with (and by means of) the $d$-homotopy category functor. This is achieved by analyzing the behavior of the $d$-homotopy category functor on inner and coCartesian fibrations (\propref{h_d_coCart_edges}). Finally, we bootstrap the results of section 2, to obtain analogues results for (essentially) $d$-operads (\thmref{d-homotopy_operad}) and some corollaries.

This work grew out of a project whose goal is to generalize the classical Eckmann-Hilton argument to the $\infty$-categorical setting. This application, which motivated the general theory we present here, will appear elsewhere.

\paragraph{Conventions.}

We work in the setting of $\infty$-categories (a.k.a. quasi-categories)
and $\infty$-operads, relying heavily on the results of \cite{HTT}
and \cite{HA}. Since we have numerous references to these two foundational
works, references to \cite{HTT} are abbreviated
as T.? and those to \cite{HA} as A.?. As a rule, we follow the notation
of \cite{HTT} and \cite{HA} whenever possible. However, we supplement this notation and deviate from it in several cases in which we believe this enhances readability:
\begin{enumerate}

\item We abuse notation by identifying an ordinary category $\mathcal{C}$
with its nerve $N\left(\mathcal{C}\right)$. 

\item We abbreviate the data of an $\infty$-operad $p\colon \mathcal{O}^{\otimes}\to\finpt$
by $\mathcal{O}$ and reserve the notation $\mathcal{O}^{\otimes}$
for the $\infty$-category that is the source of $p$. Similarly,
given two $\infty$-operads $\mathcal{O}$ and $\mathcal{U}$, we 
write $f\colon \mathcal{O}\to\mathcal{U}$ for a map of $\infty$-operads
from $\mathcal{O}$ to $\mathcal{U}$. The underlying $\infty$-category
of $\mathcal{O}$, which in \cite{HA} is denoted by $\mathcal{O}_{\left\langle 1\right\rangle }^{\otimes}$,
is here denoted by $\underline{\mathcal{O}}$.

\item Given two $\infty$-operads $\mathcal{O}$ and $\mathcal{U}$, we
denote by $\alg_{\mathcal{O}}\left(\mathcal{U}\right)$ the $\infty$-operad
$\alg_{\mathcal{O}}\left(\mathcal{U}\right)^{\otimes}\to\finpt$ from
Example A.3.2.4.4. This is the internal mapping object induced
from the closed symmetric monoidal structure on $\op$ (see A.2.2.5.13).
The underlying $\infty$-category $\underline{\alg}_{\mathcal{O}}\left(\mathcal{U}\right)$
is the usual $\infty$-category of $\mathcal{O}$-algebras in $\mathcal{U}$
(which in \cite{HA} is denoted by $\alg_{\mathcal{O}}\left(\mathcal{U}\right)$).
Moreover, the maximal Kan sub-complex $\underline{\alg}_{\mathcal{O}}\left(\mathcal{U}\right)^{\simeq}$
is the space of morphisms $\map_{\op}\left(\mathcal{O},\mathcal{U}\right)$
from $\mathcal{O}$ to $\mathcal{U}$ as objects of the $\infty$-category
$\op$.

\end{enumerate}

\section{$d$-Categories}

Recall the following definition from classical homotopy theory.
\begin{defn}
\label{def:Truncated_Space}For $d\ge0$, a space $X\in\mathcal{S}$
is called $d$-\emph{truncated} if $\pi_{i}\left(X,x\right)=0$ for all $i>d$
and all $x\in X$. 
In addition, a space is called $\left(-2\right)$-\emph{truncated} if and only if it is contractible and it is called $\left(-1\right)$-\emph{truncated} if and only if it is either contractible or empty. 
We denote by $\mathcal{S}_{\le d}$ the full subcategory of $\mathcal{S}$ spanned by the $d$-truncated spaces. 
The inclusion $\mathcal{S}_{\le d}\into\mathcal{S}$ admits a left adjoint and we call the unit of the adjunction the $d$-\emph{truncation map}.
\end{defn}

This leads to the following definition in $\infty$-category theory.
\begin{defn}
\label{def:Ess_d_Category}Let $d\ge-1$ be an integer. An \emph{essentially
$d$-category} is an $\infty$-category $\mathcal{C}$ such that for
all $X,Y\in\mathcal{C}$, the mapping space $\map_{\mathcal{C}}\left(X,Y\right)$
is $\left(d-1\right)$-truncated. We denote by $\catd$ the full subcategory
of $\cat$ spanned by essentially $d$ -categories. 
\end{defn}

\begin{ex}\label{ex:1-category}
An $\infty$-category $\mathcal{C}$ is an essentially $1$-category
if and only if it lies in the essential image of the nerve functor
$N\colon \mathbf{Cat}\to\cat$ and it is an essentially $0$-category if
and only if it is equivalent to the nerve of a poset.
\end{ex}

One might hope that for an $\infty$-category $\mathcal{C}$, the condition of being an essentially $d$-category would coincide with the condition of begin a $(d-1)$-truncated object of the presentable $\infty$-category $\cat$ in the sense of T.5.5.6.1. This turns out to be \emph{false}. The later condition is equivalent to both spaces  $\map(\Delta^0,\mathcal{C})$ and $\map(\Delta^1,\mathcal{C})$ being $(d-1)$-truncated, while the former to the $(d-1)$-truncatedness of the projection map 
\[
	\map(\Delta^1,\mathcal{C}) \to 
	\map(\Delta^{\{0\}},\mathcal{C}) \times \map(\Delta^{\{1\}},\mathcal{C}).
\]
It can be deduced that a $(d-1)$-truncated object of $\cat$ is an essentially $d$-category and that an essentially $d$-category is a $d$-truncated object of $\cat$. To see that both converses are false, consider on the one hand a $d$-truncated space as an $\infty$-groupoid, and on the other, an $\infty$-category with two objects and a $d$-truncated space of maps from the first to the second (and no other non-trivial maps).

In T.2.3.4, Lurie develops the theory of $d$-categories, which are a strict model for essentially $d$-categories. We begin by recalling some basic definitions and properties. First, we introduce the following definition/notation (which is a variation on notation T.2.3.4.11).

\begin{notation} 
\begin{lemenum} 
\item Let $A\ss B$ and $D$ be simplicial sets. We define $B\rtimes_{A}D$
by the following pushout diagram 
\[
\xymatrix{A\times D\ar[d]\ar[r] & B\times D\ar[d]\\
A\ar[r] & B\rtimes_{A}D
.}
\]
\item Let $A\ss B$ and $X$ be simplicial sets. Given two maps $f,g\colon B\to X$
such that $f|_{A}=g|_{A}$ we obtain a map $f\cup g\colon B\rtimes_{A}\partial\Delta^{1}\to X$.
A homotopy relative to $A$ (or ``rel. $A$'' for short) is an extension
of $f\cup g$ to $B\rtimes_{A}\Delta^{1}$.
\item Given inclusions of simplicial sets $A\ss B\ss C$ and a simplicial
set $X$, let $\left[B,C;X\right]$ be the set of maps $B\to X$ for
which there exists an extension to $C$. We denote by $\left[A,B,C;X\right]$
the set obtained from $\left[B,C;X\right]$ by identifying maps that
are homotopic rel. $A$.
\end{lemenum}
\end{notation}

\begin{rem}
Let $\mathcal{C}$ be an $\infty$-category, let $A\ss B$ be an inclusion
of simplicial sets, and consider $f,g\colon B\to\mathcal{C}$ such that $f|_{A}=g|_{A}$.
By the discussion at the beginning of T.2.3.4, a homotopy from $f$
to $g$ rel. $A$ is the same as an equivalence from $f$ to $g$
as objects of the $\infty$-category $\mathcal{D}$ that is given as a pullback
\[
\xymatrix{\mathcal{D}\ar[d]\ar[r] & \mathcal{C}^{B}\ar[d]\\
\Delta^{0}\ar[r] & \mathcal{C}^{A}
.}
\]

Therefore, the existence of a homotopy rel. $A$ is an \emph{equivalence
relation}. We note that the above diagram is also a \emph{homotopy}
pullback in the Joyal model structure as the right vertical map is
a categorical fibration and all objects are fibrant.
\end{rem}

\begin{defn}[T.2.3.4.1] \label{def:d_cat}
Let $\mathcal{C}$ be a simplicial set and let $d\ge -1$ be an integer. We will say that $\mathcal{C}$ is a \emph{$d$-category} if it is an $\infty$-category and the following additional conditions are satisfied:
\begin{enumerate}
	\item[(1)] Given a pair of maps $f,f'\colon \Delta^d \to \mathcal{C}$, if $f$ and $f'$ are homotopic relative to $\partial \Delta^d$, then $f=f'$.
	
	\item[(2)] Given $m>d$ and a pair of maps $f,f'\colon \Delta^m \to \mathcal{C}$, if $f\mid \partial \Delta^m = f'\mid \partial \Delta^m$, then $f=f'$.
\end{enumerate}
\end{defn}

\begin{ex}
By T.2.3.4.5, an $\infty$-category $\mathcal{C}$ is a $1$-category if and only if it is \emph{isomorphic} to the nerve of an ordinary category. By T.2.3.4.3, it is a $0$-category if and only if it is \emph{isomorphic} to the nerve of a poset (compare Example \exref{1-category})
\end{ex}

Next, we shall recall the definition of the $d$-homotopy category $h_{d}\mathcal{C}$ of an $\infty$-category $\mathcal{C}$.
Using the notation $K^{d}=\mbox{sk}^{d}K$ for the $d$-th skeleton
of a simplicial set $K$, we recall the following construction.

\begin{lem}[T.2.3.4.12] \label{lem:d_homotopy_category}
For $d\ge1$, given an $\infty$-category $\mathcal{C}$,
there exists an essentially unique simplicial set $h_{d}\mathcal{C}$,
such that for every simplicial set $K$, we have a bijection 
\[
\hom\left(K,h_{d}\mathcal{C}\right)\simeq\left[K^{d-1},K^{d},K^{d+1};\mathcal{C}\right]
\]
that is natural in $K$. We denote the canonical map by $\theta_{d}\colon \mathcal{C}\to h_{d}\mathcal{C}$.
\end{lem}

Using the above construction, we have the following definition:

\begin{defn}\label{def:d_homotopy_operad}
Given an $\infty$-category $\mathcal{C}$ and an integer $d\ge -2$, we define the $d$-homotopy category of $\mathcal{C}$ to be $h_d\mathcal{C}$ of \lemref{d_homotopy_category} when $d \ge 1$ and 
\begin{paraenum}
\item For $d=-2$ we set $h_{-2}\mathcal{C}=\Delta^{0}$.
\item For $d=-1$ we set $h_{-1}\mathcal{C}=\begin{cases}
\es & \mathcal{C}=\es\\
\Delta^{0} & \mathcal{C}\neq\es
\end{cases}$ with the unique map $\theta_{-1}\colon \mathcal{C}\to h_{-1}\mathcal{C}$.
\item For $d=0$, we first define a pre-ordered set $\tilde{h}_{0}\mathcal{C}$
with the same objects as $\mathcal{C}$ and the relation $x\le y$
if and only if $\map_{\mathcal{C}}\left(X,Y\right)\neq\es$. Then
we define $h_{0}\mathcal{C}$ to be the nerve of the poset obtained
from $\tilde{h}_{0}\mathcal{C}$ by identifying isomorphic objects.
There is a canonical map $\theta_{0}\colon \mathcal{C}\to h_{0}\mathcal{C}$
defined as the composition of $\theta_{1}\colon \mathcal{C}\to h_{1}\mathcal{C}$
with the nerve of the functor that takes each object in the homotopy
category $h_{1}\mathcal{C}$ to its class in $h_{0}\mathcal{C}$ (with the unique definition on morphisms).
\end{paraenum}
\end{defn}

\begin{war}
Note that an $\infty$-category \emph{$\mathcal{C}$} is an essentially
$d$-category if and only if all objects of $\mathcal{C}$ are $\left(d-1\right)$-truncated
in the sense of T.5.5.6.1. Hence, another way to associate an essentially
$d$-category with an $\infty$-category $\mathcal{C}$ is to consider
the full subcategory spanned by the $\left(d-1\right)$-truncated
objects. For a presentable $\infty$-category, this is denoted by $\tau_{\le d-1}\mathcal{C}$
in T.5.5.6.1 and called the \emph{$\left(d-1\right)$-truncation of}
$\mathcal{C}$. We warn the reader that the two essentially $d$-categories
$h_{d}\mathcal{C}$ and $\tau_{\le d-1}\mathcal{C}$ are usually very
different. For example, when $\mathcal{C}=\mathcal{S}$ is the $\infty$-category
of spaces, $h_{1}\mathcal{S}$ is the ordinary homotopy category of
spaces, while $\tau_{\le0}\mathcal{S}$ is equivalent to the ordinary
category of sets. 
\end{war}

The map $\theta_{d}$ has the following universal property.
\begin{lem}
\label{lem:h_d(c)_universal_property} Let $d\ge-1$ and let $\mathcal{C}$ be an $\infty$-category.
\begin{lemenum}
\item The simplicial set $h_{d}\mathcal{C}$ is a $d$-category.
\item The canonical map $\mathcal{C}\to h_{d}\mathcal{C}$ is an isomorphism
if and only if $\mathcal{C}$ is a $d$-category.
\item For every $d$-category $\mathcal{D}$, composition with the canonical
map $\mathcal{C}\to h_{d}\mathcal{C}$ induces an isomorphism of simplicial
sets
\[
\fun\left(h_{d}\mathcal{C},\mathcal{D}\right)\iso\fun\left(\mathcal{C},\mathcal{D}\right).
\]
\end{lemenum}
\end{lem}

\begin{proof}
For $d\ge1$ this is the content of T.2.3.4.12. For $d=-1$ this is
trivial. For $d=0$, (1) and (2) are obvious from the definition.
For (3) observe that we have a factorization of the map in question:
\[
\fun\left(h_{0}\mathcal{C},\mathcal{D}\right)\to\fun\left(h_{1}\mathcal{C},\mathcal{D}\right)\iso\fun\left(\mathcal{C},\mathcal{D}\right),
\]
where the second map is an isomorphism (from the claim for $d=1$). Therefore, we can assume that $\mathcal{C}$ is an ordinary category and $\mathcal{D}$ is a poset and hence both simplicial sets are discrete. 
The result now follows from the observation that every functor $\mathcal{C}\to\mathcal{D}$ factors uniquely through $h_{0}\mathcal{C}$. 
\end{proof}

Using the above results, we get the following:

\begin{prop}
\label{prop:h_d_left_adjoint} The inclusion $\catd\into\cat$ admits
a left adjoint $h_{d}\colon \cat\to\catd$ with unit map given by $\theta_{d}\colon \mathcal{C}\to h_{d}\mathcal{C}$.
\end{prop}

\begin{proof}
By T.2.3.4.18, every essentially $d$-category is equivalent to a
$d$-category and for every $d$-category $\mathcal{D}$, the map 
\[
\fun\left(h_{d}\mathcal{C},\mathcal{D}\right)\to\fun\left(\mathcal{C},\mathcal{D}\right)
\]
is an isomorphism by \lemref{h_d(c)_universal_property}. Restricting to the maximal Kan sub-complexes, the map of simplicial sets 

\[
\theta_{d}^{*}\colon \map_{\catd}\left(h_{d}\mathcal{C},\mathcal{D}\right)\to\map_{\cat}\left(\mathcal{C},\mathcal{D}\right)
\]
is a homotopy equivalence. It now follows that $\theta_{d}$ exhibits $h_{d}\mathcal{C}$ as the
$\catd$-localization of $\mathcal{C}$ in the sense of T.5.2.7.6. Thus, the claim about the existence of a left adjoint follows
from T.5.2.7.8 and the claim about the unit follows from the proof of T.5.2.7.8.
\end{proof}
The main goal of this section is to show that for every $\infty$-category $\mathcal{C}$, the $d$-category $h_{d}\mathcal{C}$ is obtained (as one would expect) by $\left(d-1\right)$-truncation of the mapping spaces. 
The main ingredient is the following explicit description of the right mapping space in the $d$-homotopy category. 
\begin{prop}
\label{prop:h_d_mapping_spaces_cat} Let $d\ge-1$ and let $\mathcal{C}$
be an $\infty$-category. For every $X,Y\in\mathcal{C}$, there is
a canonical isomorphism $\alpha$ of simplicial sets rendering the
following diagram commutative:
\[
\xymatrix{ & \hom_{\mathcal{C}}^{R}\left(X,Y\right)\ar[ld]_{\beta}\ar[rd]^{\gamma}\\
\hom_{h_{d}\mathcal{C}}^{R}\left(\theta_{d}\left(X\right),\theta_{d}\left(Y\right)\right)\ar[rr]_{\alpha}^{\sim} &  & h_{d-1}\hom_{\mathcal{C}}^{R}\left(X,Y\right),
}
\]
where $\beta$ and $\gamma$ are the obvious maps.
\end{prop}

We defer the rather technical proof of \propref{h_d_mapping_spaces_cat}
to the end of the section. Assuming \propref{h_d_mapping_spaces_cat},
we get
\begin{cor}
\label{cor:Theta_d_Mapping_Space_Cat}Let $d\ge-1$ and let $\mathcal{C}$
be an $\infty$-category. The canonical map $\theta_{d}\colon \mathcal{C}\to h_{d}\mathcal{C}$
is essentially surjective and for every $X,Y\in\mathcal{C}$, the
induced map
\[
\map_{\mathcal{C}}\left(X,Y\right)\to\map_{h_{d}\mathcal{C}}\left(\theta_{d}\left(X\right),\theta_{d}\left(Y\right)\right)
\]
 is a $\left(d-1\right)$-truncation map. 
\end{cor}

\begin{proof}
It is clear that $\theta_{d}$ is essentially surjective since it
is surjective on objects. Let $X,Y\in\mathcal{C}$ be two objects.
Since the map 
\[
\map_{\mathcal{C}}\left(X,Y\right)\to\map_{h_{d}\mathcal{C}}\left(\theta_{d}\left(X\right),\theta_{d}\left(Y\right)\right)
\]
is represented by the map
\[
\theta\colon \hom_{\mathcal{C}}^{R}\left(X,Y\right)\to h_{d-1}\hom_{\mathcal{C}}^{R}\left(X,Y\right),
\]
it will be enough to show that for every Kan complex $X$, the map
$X\to h_{d-1}X$ is a $\left(d-1\right)$-truncation map. We prove
this by induction. For $d\le0$ it is clear. For $d\ge1$, recall
that $\hom_{X}^{R}\left(p,q\right)$ has the homotopy type of the
path space $P_{p,q}X$ between $p$ and $q$ in $X$ when viewed as a space.
Thus, by induction, $\theta$ is a map of spaces that is surjective
on $\pi_{0}$ and induces the $\left(d-2\right)$-truncation map on path spaces
\[
P_{p,q}X\to P_{p,q}\left(h_{d-1}X\right)\simeq h_{d-2}\left(P_{p,q}X\right).
\]
It follows that $\theta$ is a $\left(d-1\right)$-truncation map.
\end{proof}
\begin{thm}
\label{thm:d-homotopy_category}The inclusion functor $\catd\into\cat$ admits
a left adjoint $h_{d}$ such that for every $\infty$-category $\mathcal{C}$, the value of $h_d$ on $\mathcal{C}$  is the $d$-homotopy category of $\mathcal{C}$, the unit transformation $\theta_{d}\colon \mathcal{C}\to h_{d}\mathcal{C}$
is essentially surjective, and for all $X,Y\in\mathcal{C}$, the map of spaces
\[
\map_{\mathcal{C}}\left(X,Y\right)\to\map_{h_{d}\mathcal{C}}\left(\theta_{d}\left(X\right),\theta_{d}\left(Y\right)\right)
\]
is the $\left(d-1\right)$-truncation map.
\end{thm}

\begin{proof}
Combine \propref{h_d_left_adjoint} and \corref{Theta_d_Mapping_Space_Cat}.
\end{proof}
To prove \propref{h_d_mapping_spaces_cat}, we begin by recalling the
definitions of the ``right'' and ``middle'' mapping spaces.
Let $J\colon \sset\to\sset_{\partial\Delta^{1}/}$ be the functor given
by $J\left(K\right)=K\star\Delta^{0}/K$, with the natural map $\partial\Delta^{1}\to J\left(K\right)$
taking $0$ to the image of $K$ and $1$ to the cone point. Recall
that by the definition of the right mapping space (right before T.1.2.2.3),
we have 
\[
\hom(\Delta^{n},\hom_{\mathcal{C}}^{R}\left(X,Y\right))=
\hom_{(X,Y)}(J(\Delta^{n}),\mathcal{C}),
\]
where the subscript $\left(X,Y\right)$ in the right hand side means
we take the subset of maps that restrict to $\left(X,Y\right)$ on
$\partial\Delta^{1}$. Since $J$ preserves colimits, it follows that
for every simplicial set $K$, we have a canonical isomorphism
\[
\hom(K,\hom_{\mathcal{C}}^{R}\left(X,Y\right))=
\hom_{(X,Y)}(J(K),\mathcal{C}).
\]
Similarly, we can construct the ``middle mapping space''. Let $\Sigma\colon \sset\to\sset$
be the functor given by $\Sigma\left(K\right)=K\diamond\Delta^{0}/K$.
This also comes with a canonical map $\partial\Delta^{1}\to\Sigma\left(K\right)$,
and similarly, from the definition of the middle mapping space (right
after remark T.1.2.2.5), we have 
\[
\hom(K,\hom_{\mathcal{C}}^{M}(X,Y))=\hom_{(X,Y)}(\Sigma(K),\mathcal{C}).
\]
There is a canonical categorical equivalence $K\diamond\Delta^{0}\iso K\star\Delta^{0}$
that induces a categorical equivalence $\Sigma K\to J\left(K\right)$
that induces a Kan equivalence
\[
\Phi\colon \hom_{\mathcal{C}}^{R}\left(X,Y\right)\iso\hom_{\mathcal{C}}^{M}\left(X,Y\right)
\]
of Kan complexes. 

For $f\colon K\to\hom_{\mathcal{C}}^{R}\left(X,Y\right)$, we denote by
$\overline{f}\colon J\left(K\right)\to\mathcal{C}$ the corresponding map
in the definition of $\hom_{\mathcal{C}}^{R}\left(X,Y\right)$. We
also denote by $F=\Phi\circ f$ and $\overline{F}\colon \Sigma\left(K\right)\to\mathcal{C}$
the corresponding map in the definition of $\hom_{\mathcal{C}}^{M}\left(X,Y\right)$.
We begin with the following technical lemma.
\begin{lem}
\label{lem:Homotopy_Cylinder} Given simplicial sets $A\ss B$ and
$D$, there is a canonical isomorphism 
\[
\Sigma\left(B\rtimes_{A}D\right)\iso\Sigma B\rtimes_{\Sigma A}D.
\]
\end{lem}

\begin{proof}
Consider the following diagram (with the obvious maps) and compute
the colimit, starting once with the rows and once with the columns:
\[
\xymatrix{\partial\Delta^{1} & \partial\Delta^{1}\times D\ar[r]\ar[l] & \partial\Delta^{1}\times D & \partial\Delta^{1}\times\left(\Delta^0\rtimes_{\Delta^0}D\right)\\
\partial\Delta^{1}\times A\ar[d]\ar[u] & \partial\Delta^{1}\times A\times D\ar[r]\ar[l]\ar[d]\ar[u] & \partial\Delta^{1}\times B\times D\ar[d]\ar[u] & \partial\Delta^{1}\times\left(B\rtimes_{A}D\right)\ar[d]\ar[u]\\
\Delta^{1}\times A & \Delta^{1}\times A\times D\ar[r]\ar[l] & \Delta^{1}\times B\times D & \Delta^{1}\times\left(B\rtimes_{A}D\right)\\
\Sigma A & \Sigma A\times D\ar[r]\ar[l] & \Sigma B\times D & \Sigma B\rtimes_{\Sigma A}D\simeq\Sigma\left(B\rtimes_{A}D\right)
}
.
\]
\end{proof}
The following lemma compares the different models of the mapping space.
\begin{lem}
\label{lem:Homotopy_Rel_A}Given simplicial sets $A\ss B$ and two
maps $f,g\colon B\to\hom_{\mathcal{C}}^{R}\left(X,Y\right)$, the following
are equivalent:
\begin{lemenum}
\item $f,g\colon B\to\hom_{\mathcal{C}}^{R}\left(X,Y\right)$ agree on $A$ (resp.
homotopic rel. $A$).
\item $F,G\colon B\to\hom_{\mathcal{C}}^{M}\left(X,Y\right)$ agree on $A$ (resp.
homotopic rel. $A$).
\item $\overline{f},\overline{g}\colon J\left(B\right)\to C$ agree on $J\left(A\right)$
(resp. homotopic rel. $J\left(A\right)$).
\item $\overline{F},\overline{G}\colon \Sigma\left(B\right)\to C$ agree on $\Sigma\left(A\right)$
(resp. homotopic rel. $\Sigma\left(A\right)$).
\end{lemenum}
\end{lem}
\begin{proof}
We start with the equivalence $\left(1\right)\iff\left(2\right)$. The
first part follows from the fact that $\Phi$ is a monomorphism and
the second part follows from the fact that $\Phi$ is a homotopy equivalence of Kan complexes. 
In the equivalence $\left(3\right)\iff\left(4\right)$,
the first part follows from the fact that $\Sigma A\to J\left(A\right)$
is an epimorphism and the second part can be seen as follows: the
maps $\overline{f},\overline{g}\colon J\left(B\right)\to\mathcal{C}$ are
homotopic rel $J\left(A\right)$ if and only if they are equivalent
as elements of the $\infty$-category that is the fiber over $\overline{f}|_{J\left(A\right)}=\overline{g}|_{J\left(A\right)}$
(which is also a homotopy fiber) of the categorical fibration $\mathcal{C}^{J\left(B\right)}\to\mathcal{C}^{J\left(A\right)}$.
Since we have functorial categorical equivalences $\Sigma\left(A\right)\iso J\left(A\right)$
and $\Sigma\left(B\right)\iso J\left(B\right)$, this is the same
as showing that the corresponding maps $\overline{F},\overline{G}\colon \Sigma\left(B\right)\to\mathcal{C}$
are equivalent in the fiber of $\mathcal{C}^{\Sigma\left(B\right)}\to\mathcal{C}^{\Sigma\left(A\right)}$
(which is also the homotopy fiber). This in turn is the same as having
$\overline{F},\overline{G}$ homotopic rel. $\Sigma A$. It is left
to show the equivalence $\left(2\right)\iff\left(4\right)$. The
first part is clear. The second part amounts to showing the equivalence of two extension problems. 
If $F|_{A}=G|_{A}$, we get a map $F\cup_{A}G$ from 
$B\cup_{A}B\simeq B\rtimes_{A}\partial\Delta^{1}$ to 
$\hom_{\mathcal{C}}^{M}\left(X,Y\right)$
and $F$ and $G$ are homotopic rel. $A$ if and only if $F\cup_{A}G$
extends to the relative cylinder $B\rtimes_{A}\Delta^{1}$. In terms
of maps to $\mathcal{C}$, this is equivalent to the extension problem
\[
\xymatrix{\Sigma\left(B\rtimes_{A}\partial\Delta^{1}\right)\ar[r]\ar[d] & \mathcal{C}\\
\Sigma\left(B\rtimes_{A}\Delta^{1}\right)\ar@{-->}[ru]
}
.
\]
On the other hand, from $\overline{F}|_{\Sigma A}=\overline{G}|_{\Sigma A}$
we get a map $\overline{F}\cup_{\Sigma A}\overline{G}$ from $\Sigma B\rtimes_{\Sigma A}\partial\Delta^{1}$
to $\mathcal{C}$ and $\overline{F}$ and $\overline{G}$ are homotopic
rel. $\Sigma A$ if and only if it extends to the relative cylinder
$\Sigma B\rtimes_{\Sigma A}\Delta^{1}$. By \lemref{Homotopy_Cylinder}
for $D=\Delta^{1},\partial\Delta^{1}$, the two extension problems
are isomorphic.
\end{proof}
We are now ready to prove \propref{h_d_mapping_spaces_cat}.
\begin{proof}[Proof (of \propref{h_d_mapping_spaces_cat})]
For $d\le0$ this follows directly from the definitions, and so we assume that $d\ge1$. Let $K$ be a simplicial set. On the one hand, 
\begin{eqnarray*}
\hom\left(K,\hom_{h_{d}\mathcal{C}}^{R}\left(X,Y\right)\right) & = & \hom_{\left(X,Y\right)}\left(J\left(K\right),h_{d}\mathcal{C}\right)\\
 & = & \left[J\left(K\right)^{d-1},J\left(K\right)^{d},J\left(K\right)^{d+1};\mathcal{C}\right]_{\left(X,Y\right)},
\end{eqnarray*}
where subscript $\left(X,Y\right)$ indicates that we take only
the subset of maps that restrict to $\left(X,Y\right)$ on $\partial\Delta^{1}\into J\left(K\right)$
(observe that this is independent of the representative as $\partial\Delta^{1}\ss J\left(K\right)^{d-1}$).
On the other hand, 
\begin{eqnarray*}
\hom\left(K,h_{d-1}\hom_{\mathcal{C}}^{R}\left(X,Y\right)\right) & = & \left[K^{d-2},K^{d-1},K^{d};\hom_{\mathcal{C}}^{R}\left(X,Y\right)\right].
\end{eqnarray*}
We will argue that this last set is in natural bijection with the
set 
\[
\left[J\left(K^{d-2}\right),J\left(K^{d-1}\right),J\left(K^{d}\right),\mathcal{C}\right]_{\left(X,Y\right)}.
\]
First, by definition of the right mapping space we have a natural
bijection between maps of the form $f\colon K^{d-1}\to\hom_{\mathcal{C}}^{R}\left(X,Y\right)$ and maps of the form $\overline{f}\colon J\left(K^{d-1}\right)\to\mathcal{C}$ 
restricting
to $\left(X,Y\right)$ on $\partial\Delta^{1}\ss J\left(K^{d-1}\right)$.
Second, $f$ extends to $K^{d}$ if and only if $\overline{f}$ extends
to $J\left(K^{d}\right)$. Likewise, it is clear that two maps $f,g\colon K^{d-1}\to\hom_{\mathcal{C}}^{R}\left(X,Y\right)$
\emph{agree} on $K^{d-2}$ if and only if the corresponding maps $\overline{f},\overline{g}\colon J\left(K^{d-1}\right)\to\mathcal{C}$
\emph{agree} on $J\left(K^{d-2}\right)$. Hence, the only thing we
need to show is that $f$ and $g$ are homotopic rel. $K^{d-2}$ if
and only if $\overline{f}$ and $\overline{g}$ are homotopic rel.
$J\left(K^{d-2}\right)$ and this follows from $\left(1\right)\iff\left(3\right)$ in \lemref{Homotopy_Rel_A}. 
It remains to observe that for every simplicial set $K$ and every $d\ge1$ we have a canonical
isomorphism $J\left(K^{d-1}\right)\iso J\left(K\right)^{d}.$ Hence,
we get a natural bijection 
\[
\hom\left(K,\hom_{h_{d}\mathcal{C}}^{R}\left(X,Y\right)\right)\simeq\hom\left(K,h_{d-1}\hom_{\mathcal{C}}^{R}\left(X,Y\right)\right)
\]
and therefore an isomorphism $\alpha\colon \hom_{h_{d}\mathcal{C}}^{R}\left(X,Y\right)\simeq h_{d-1}\hom_{\mathcal{C}}^{R}\left(X,Y\right)$. 

Finally, we need to show that the isomorphism we have constructed is compatible with
the maps $\theta\colon \hom_{\mathcal{C}}^{R}\left(X,Y\right)\to h_{d-1}\hom_{\mathcal{C}}^{R}\left(X,Y\right)$ and 
$\beta\colon \hom_{\mathcal{C}}^{R}\left(X,Y\right)\to\hom_{h_{d}\mathcal{C}}^{R}\left(X,Y\right)$. 
For this, consider a map $f\colon K\to\hom_{\mathcal{C}}^{R}\left(X,Y\right)$. 
The composition $\theta\circ f$ is represented by the restriction $f|_{K^{d-1}}$,
which corresponds to the map $\overline{f|_{K^{d-1}}}\colon J\left(K^{d-1}\right)\to\mathcal{C}$.
On the other hand, the composition $\beta\circ f$ corresponds to
the restriction of $\overline{f}\colon J\left(K\right)\to\mathcal{C}$ to
$J\left(K\right)^{d+1}$ and these are identified by $\alpha$.
\end{proof}

\section{$d$-Operads}

We now develop the basic theory of (essentially) $d$-operads in analogy
with (and by bootstrapping of) the theory of $d$-categories. First,
\begin{defn}\label{def:Ess_d_Operad}
Let $d\ge-1$. An \emph{essentially $d$-operad} is an $\infty$-operad
$\mathcal{O}$ such that for all $X_{1},\dots,X_{n},Y\in\underline{\mathcal{O}}$,
the multi-mapping space $\mbox{Mul}_{\mathcal{O}}\left(\left\{ X_{1},\dots,X_{n}\right\} ;Y\right)$
is $\left(d-1\right)$-truncated. We denote by $\opd$ the full subcategory
of $\op$ spanned by essentially $d$ -operads.
\end{defn}

\begin{ex}
\label{exa:Sym_Mon_d_Cat} Two important special cases are:
\begin{paraenum}
\item A symmetric monoidal $\infty$-category is an essentially
$d$-operad if and only if its underlying $\infty$-category is an essentially $d$-category.
\item A reduced $\infty$-operad $\mathcal{P}$ is an essentially $d$-operad
if and only if the corresponding symmetric sequence of $n$-ary operations $\left\{ \mathcal{P}\left(n\right)\right\} _{n\ge0}$ consists of $\left(d-1\right)$-truncated spaces.
\end{paraenum}
\end{ex}
We begin by showing that that the functor $h_{d}$ behaves well with
respect to inner and coCartesian edges.
\begin{prop}
\label{prop:h_d_coCart_edges} Let $d\ge-1$ and let $p\colon \mathcal{C}\to\mathcal{D}$
be a functor, where $\mathcal{C}$ is an $\infty$-category and $\mathcal{D}$
a $d$-category. 
\begin{lemenum}
\item If the functor $p\colon \mathcal{C}\to\mathcal{D}$ is an inner fibration, then so is $h_{d}\left(p\right)\colon h_{d}\left(\mathcal{C}\right)\to h_{d}\left(\mathcal{D}\right)=\mathcal{D}$.
\item If in addition $f$ is a $p$-coCartesian morphism in $\mathcal{C}$,
then $h_{d}\left(f\right)$ is $h_{d}\left(p\right)$-coCartesian
in $h_{d}\mathcal{C}$.
\end{lemenum}
\end{prop}

\begin{proof}
For $d=-1,0$, both assertions are trivial to check and so we assume that $d\ge1$.
The argument that $h_{d}\left(p\right)$ is an inner fibration is similar to the argument that $h_{d}\left(f\right)$ is coCartesian and so we shall prove
them together. Using T.2.4.1.4, we need to consider the lifting problem

\[
\LP{\Lambda_{i}^{m}}{\Delta^{m}}{h_{d}\mathcal{C}}{\mathcal{D}}
\]
for some $m\ge2$ and either 
\begin{paraenum}
\item $0<i<m$ or 
\item $i=0$ and $\Delta^{\left\{ 0,1\right\} }\ss\Lambda_{0}^{m}$ is mapped
in $h_{d}\mathcal{C}$ to $h_{d}\left(f\right)$. 
\end{paraenum}
For $m\ge d+3$, we have $\mbox{sk}^{j}\Lambda_{i}^{m}=\mbox{sk}^{j}\Delta^{m}$
for all $j\le d+1$, and so the map 
\[
\hom\left(\Delta^{m},h_{d}\mathcal{C}\right)\to\hom\left(\Lambda_{i}^{m},h_{d}\mathcal{C}\right)
\]
is a bijection and there is nothing to prove. For $m\le d+2$, we
have $\Lambda_{i}^{m}=\mbox{sk}^{d+1}\Lambda_{i}^{m}$, and so the map
\[
\hom\left(\Lambda_{i}^{m},\mathcal{C}\right)\onto\hom\left(\Lambda_{i}^{m},h_{d}\mathcal{C}\right)
\]
is surjective, hence the map $\Lambda_{i}^{m}\to h_{d}\mathcal{C}$
factors through $\Lambda_{i}^{m}\to\mathcal{C}$. Now, the functor
$\mathcal{C}\to h_{d}\mathcal{C}$ identifies only homotopic morphisms
(for $d\ge1$); hence in (2) the image of $\Delta^{\left\{ 0,1\right\}}$
in $\mathcal{C}$ is coCartesian. Thus, in both cases we can solve the
corresponding lifting problem in $\mathcal{C}$, which induces a lift
in the original square.
\end{proof}
\begin{defn}\label{def:d_operad}
Let $\mathcal{O}$ be an $\infty$-operad.
\begin{paraenum}
\item For $d\ge1$, we say that $\mathcal{O}$ is a \emph{$d$-operad} if $\mathcal{O}^{\otimes}$
is a $d$-category. 
\item We say that $\mathcal{O}$ is a \emph{$0$-operad} if $\mathcal{O}^{\otimes}$ is a skeletal
$1$-category and $p$ is faithful.
\item We say that $\mathcal{O}$ is a \emph{$(-1)$-operad} if either $\mathcal{O}^{\otimes}=\es$
or $p$ is an isomorphism.
\end{paraenum}
\end{defn}

\begin{rem}
A $d$-operad is intended to bear the same relation to an essentially
$d$-operad as a $d$-category does to an essentially $d$-category; i.e.
it is a strict model for an $\infty$-operad in which all multi-mapping
spaces are $\left(d-1\right)$-truncated.
\end{rem}

Next, we define the notion of a $d$-homotopy operad of an $\infty$-operad,
which is analogous to the notion of a $d$-homotopy category of an
$\infty$-category. 
\begin{defn}\label{def:d_homotopy_operad}
Given an $\infty$-operad $p\colon \mathcal{O}^{\otimes}\to\finpt$, we
define its \emph{$d$-homotopy operad} $h_{d}\mathcal{O}$ to be a
map of simplicial sets $p_{d}\colon \left(h_{d}\mathcal{O}\right)^{\otimes}\to\finpt$
defined as follows:
\begin{paraenum}
\item For $d\ge1$, we simply apply $h_{d}$ to $p$ as a functor between
$\infty$-categories and use the fact that $\finpt$ is a $1$-category;
hence there is a canonical isomorphism $h_{d}\left(\finpt\right)\simeq\finpt$.
\item For $d=0$, we first construct the (ordinary) category $\tilde{h}_{0}\mathcal{O}^{\otimes}$
whose objects are those of $\mathcal{O}^{\otimes}$ and each mapping
space is replaced by its image in $\finpt$. Then we identify isomorphic
objects in $\tilde{h}_{0}\mathcal{O}^{\otimes}$ (note that there
is a unique induced composition, since isomorphic objects are mapped
to the same object in $\finpt$) and finally we define $\left(h_{0}\mathcal{O}\right)^{\otimes}$
to be the nerve of the resulting category, with $p_{0}$ being the obvious
map to $\finpt$.
\item For $d=-1$, we define $p_{d}\colon \finpt\to\finpt$ to be the identity functor if $\mathcal{O}^{\otimes}\neq\es$ and the unique functor $p_{d}\colon \es\to\finpt$ otherwise.
\end{paraenum}
\end{defn}

In all three cases we have a canonical map of simplicial sets $\theta_{d}\colon \mathcal{O}^{\otimes}\to\left(h_{d}\mathcal{O}\right)^{\otimes}$
over $\finpt$.

\begin{war}
For every $\infty$-operad $\mathcal{O}$ and $d\ge1$
we have $\left(h_{d}\mathcal{O}\right)^{\otimes}\simeq h_{d}\left(\mathcal{O}^{\otimes}\right)$,
but for $d\le0$ we get something slightly different. The reason for this is
that $\left(h_{d}\mathcal{C}\right)^{\otimes}$ corresponds to the
application of $h_{d}$ \emph{fiber-wise} to the map $p\colon \mathcal{O}^{\otimes}\to\finpt$.
Since $\finpt$ is a 1-category, for $d\ge1$ this is the same as
applying $h_{d}$ to $p$, but for $d\le0$ it is not.
\end{war}

\begin{lem}
Let $p\colon \mathcal{O}^{\otimes}\to\finpt$ be an $\infty$-operad.

\begin{lemenum}
\item The map $p_{d}\colon \left(h_{d}\mathcal{O}\right)^{\otimes}\to\finpt$
is a $d$-operad.
\item The canonical map $\theta_{d}\colon \mathcal{O}\to h_{d}\mathcal{O}$ is a map of $\infty$-operads.
\item Given an $\infty$-operad map $F\colon \mathcal{O}\to\mathcal{U}$, the induced map $h_{d}F\colon h_{d}\mathcal{O}\to h_{d}\mathcal{U}$ on $d$-homotopy operads, is an $\infty$-operad map.
\end{lemenum}
\end{lem}

\begin{proof}
For $d=-1$, there is nothing to prove in (1)\textendash (3) and so we assume that $d\ge0$. 

(1) For $d=0$, it is clear that $\left(h_{0}\mathcal{\mathcal{O}}\right)^{\otimes}$
is a skeletal $1$-category, with $p_{0}$ fully faithful; and for $d\ge1$,
it is clear that $\left(h_{d}\mathcal{\mathcal{O}}\right)^{\otimes}$
is a $d$-category. Hence, we only need to show that $\left(h_{d}\mathcal{\mathcal{O}}\right)^{\otimes}$
is an $\infty$-operad. For this we need to check the three conditions
of Definition A.2.1.1.10. 
\begin{itemize}
\item Since $p\colon \mathcal{O}^{\otimes}\to\finpt$ is an $\infty$-operad,
for every inert morphism $f\colon \left\langle m\right\rangle \to\left\langle n\right\rangle $
and an object $\overline{X}\in h_{d}\mathcal{O}_{\left\langle m\right\rangle }^{\otimes}$,
we can lift $\overline{X}$ to $X\in\mathcal{O}_{\left\langle m\right\rangle }^{\otimes}$
and find a coCartesian lift $g\colon X\to Y$ of $f$ in $\mathcal{O}^{\otimes}$.
For $d\ge1$, the image $\overline{g}$ of $g$ in $\left(h_{d}\mathcal{O}\right)^{\otimes}$
is a coCartesian lift of $f$ by \propref{h_d_coCart_edges}. For $d=0$,
we use the dual of T.2.4.4.3 to show that $\overline{g}$ is coCartesian.
$\left(h_{0}\mathcal{O}\right)^{\otimes}\to\finpt$ is an inner fibration
(as the nerve of a functor of ordinary categories) and for every $\overline{Z}\in\left(h_{0}\mathcal{O}\right)_{\left\langle m\right\rangle }^{\otimes}$,
pre-composition with $\overline{g}$ induces a diagram 
\[
\xymatrix{\map_{\left(h_{0}\mathcal{O}\right)^{\otimes}}\left(\overline{Y},\overline{Z}\right)\ar[d]\ar[r] & \map_{\left(h_{0}\mathcal{O}\right)^{\otimes}}\left(\overline{X},\overline{Z}\right)\ar[d]\\
\map_{\finpt}\left(\left\langle m\right\rangle ,\left\langle k\right\rangle \right)\ar[r] & \map_{\finpt}\left(\left\langle n\right\rangle ,\left\langle k\right\rangle \right)
}
,
\]
and it is easy to verify that it is a homotopy pullback.
\item Let $\overline{X}\in\left(h_{d}\mathcal{O}\right)_{\left\langle m\right\rangle }^{\otimes}$
and $\overline{Y}\in\left(h_{d}\mathcal{O}\right)_{\left\langle n\right\rangle }^{\otimes}$
and let 
$f\colon \left\langle m\right\rangle \to\left\langle n\right\rangle$
be a morphism in $\finpt$. We first observe that 
\[
\map_{\left(h_{d}\mathcal{O}\right)^{\otimes}}^{f}\left(X,Y\right)\simeq h_{d-1}\left(\map_{\mathcal{O}^{\otimes}}^{f}\left(X,Y\right)\right).
\]
For $d\ge1$ this follows from \propref{h_d_mapping_spaces_cat} and
for $d=0$ it follows directly from the definition. Hence,
\begin{eqnarray*}
\map_{\left(h_{d}\mathcal{O}\right)^{\otimes}}^{f}\left(X,Y\right) & \simeq & h_{d-1}\left(\map_{\mathcal{O}^{\otimes}}^{f}\left(X,Y\right)\right)\simeq h_{d-1}\left(\prod_{1\le i\le n}\map_{\mathcal{O}^{\otimes}}^{\rho^{i}\circ f}\left(X,Y_{i}\right)\right)\\
 & \simeq & \prod_{1\le i\le n}h_{d-1}\left(\map_{\mathcal{O}^{\otimes}}^{\rho^{i}\circ f}\left(X,Y_{i}\right)\right)\simeq\prod_{1\le i\le n}\map_{\left(h_{d}\mathcal{O}\right)^{\otimes}}^{\rho^{i}\circ f}\left(X,Y_{i}\right).
\end{eqnarray*}

Note that we use the fact that $h_{d}$ preserves finite products of spaces.
\item For every finite collection of objects $\overline{X}_{1},\dots,\overline{X}_{n}\in\left(h_{d}\mathcal{O}\right)_{\left\langle 1\right\rangle }^{\otimes}$
that are lifted to objects of $\mathcal{O}_{\left\langle 1\right\rangle }^{\otimes}$,
there is an object $X\in\mathcal{O}_{\left\langle n\right\rangle }^{\otimes}$
and coCartesian morphisms $f_{i}\colon X\to X_{i}$ covering $\rho^{i}\colon \left\langle n\right\rangle \to\left\langle 1\right\rangle $.
The images of those maps in $h_{d}\mathcal{O}^{\otimes}$ are coCartesian
as well and satisfy the analogous property.
\end{itemize}
(2) From the proof of (1), $\theta_{d}$ maps inert morphisms in $\mathcal{O}^{\otimes}$
to inert morphisms in $h_{d}\mathcal{O}^{\otimes}$.

(3) We need to show that $h_{d}F$ maps inert morphisms to inert morphisms.
For $d=0$, this is automatic. For $d\ge1$, let $\overline{f}\colon X\to Y$
be an inert morphism in $\left(h_{d}\mathcal{O}\right)^{\otimes}$.
There is a coCartesian morphism $f\colon X\to Y'$ in $\mathcal{O}^{\otimes}$
with the same image as $\overline{f}$ in $\finpt$; hence its image
in $\left(h_{d}\mathcal{O}\right)^{\otimes}$ is equivalent to $f$.
Since the composition $\mathcal{O}^{\otimes}\to\mathcal{U}^{\otimes}\to\left(h_{d}\mathcal{U}\right)^{\otimes}$
preserves inert morphisms, it follows that the image of $f$ in $\left(h_{d}\mathcal{U}\right)^{\otimes}$
is inert and since the image of $\overline{f}$ in $\left(h_{d}\mathcal{U}\right)^{\otimes}$
is equivalent to the image of $f$, it is inert as well.
\end{proof}

The following lemma provides the universal property of $\theta_d$ by analogy with \lemref{h_d(c)_universal_property} for $d$-categories.
\begin{lem}
\label{lem:theta_d_universal_property_operad}Let $\mathcal{O}$ be
an $\infty$-operad.
\begin{lemenum}
\item $\mathcal{O}$ is a $d$-operad if and only if $\theta_{d}$ is an
isomorphism.
\item For every $d$-operad $\mathcal{U}$, pre-composition with $\theta_{d}$
induces an isomorphism of simplicial sets
\[
\underline{\alg}_{h_{d}\mathcal{O}}\left(\mathcal{U}\right)\to\underline{\alg}_{\mathcal{O}}\left(\mathcal{U}\right)
\]
and in particular a homotopy equivalence 
\[
\map_{\op}\left(h_{d}\mathcal{O},\mathcal{U}\right)\to\map_{\op}\left(\mathcal{O},\mathcal{U}\right).
\]
\end{lemenum}
\end{lem}
\begin{proof}
(2) Assume that $d\ge1$. By the analogous fact for $\infty$-categories,
the composition with $\theta_{d}$ induces an isomorphism 
\[
\fun_{\finpt}((h_{d}\mathcal{O})^{\otimes},\mathcal{U}^{\otimes})\iso\fun_{\finpt}(\mathcal{O}^{\otimes},\mathcal{U}^{\otimes}).
\]
The simplicial set $\underline{\alg}_{\mathcal{O}}\left(\mathcal{U}\right)$
is the full subcategory of 
$\fun_{\finpt}\left(\mathcal{O}^{\otimes},\mathcal{U}^{\otimes}\right)$
spanned by maps of $\infty$-operads (and similarly for $h_{d}\mathcal{O}$
instead of $\mathcal{O}$). The claim now follows from the fact that
the image of a coCartesian edge in $\mathcal{O}^{\otimes}$ is coCartesian
in $\left(h_{d}\mathcal{O}\right)^{\otimes}$ and, conversely, every
inert morphism in $\left(h_{d}\mathcal{O}\right)^{\otimes}$ is up
to equivalence the image of an inert morphism in $\mathcal{O}^{\otimes}$
(lift the source to some object $X\in\mathcal{O}^{\otimes}$ and choose
any inert map with domain $X$).

For $d=0$, essentially the same argument works, only now the inert
maps of $\left(h_{0}\mathcal{O}\right)^{\otimes}$ are precisely those
whose image in $\finpt$ is inert and therefore the inert maps of
$\left(h_{0}\mathcal{O}\right)^{\otimes}$ are again precisely the
images of inert maps in $\mathcal{O}^{\otimes}$. For $d=-1$, the
claim is obvious.

(1) Follows from (2) and the Yoneda lemma in the 1-category $\pop$ of $\infty$-preoperads (see A.2.1.4.2).
\end{proof}
\begin{lem}
\label{lem:h_d_mapping_spaces_operads}Let $d\ge-1$ and let $\mathcal{O}$ be
an $\infty$-operad. The canonical map $\theta_{d}\colon \mathcal{O}\to h_{d}\mathcal{O}$
is essentially surjective and for all $X_{1},\dots,X_{n},Y\in \underline{\mathcal{O}}$,
the map
\[
\mul_{\mathcal{O}}\left(\left\{ X_{1},\dots,X_{n}\right\} ;Y\right)\to\mul_{h_{d}\mathcal{O}}\left(\left\{ \theta_{d}\left(X_{1}\right),\dots,\theta_{d}\left(X_{n}\right)\right\} ;\theta_{d}\left(Y\right)\right)
\]
is a $\left(d-1\right)$-truncation map.
\end{lem}

\begin{proof}
The map $\theta_{d}\colon \mathcal{O}\to h_{d}\mathcal{O}$ is surjective
on objects and hence is essentially surjective. For $d\ge1$,
the second assertion follows from the corresponding fact for $\infty$-categories; and for $d=-1,0$, it follows directly from the definition.
\end{proof}
\begin{cor}
\label{cor:char_of_ess_d-operad}An $\infty$-operad \emph{is an essentially
$d$-operad if and only if it is equivalent to a $d$-operad.}
\end{cor}

The following is the analogue of \thmref{d-homotopy_category} for $\infty$-operads.
\begin{thm}
\label{thm:d-homotopy_operad} The inclusion $\opd\into\op$ admits
a left adjoint $h_{d}$, such that for every $\infty$-operad $\mathcal{O}$ the value of $h_d$ on $\mathcal{O}$  is the $d$-homotopy operad of $\mathcal{O}$, the unit transformation $\theta_{d}\colon \mathcal{O}\to h_{d}\mathcal{O}$ is essentially surjective, and for all objects
$X_{1},\dots,X_{n},Y\in \underline{\mathcal{O}}$,
the map of spaces
\[
\mul_{\mathcal{O}}\left(\left\{ X_{1},\dots,X_{n}\right\} ;Y\right)\to\mul_{h_{d}\mathcal{O}}\left(\left\{ \theta_{d}\left(X_{1}\right),\dots,\theta_{d}\left(X_{n}\right)\right\} ;\theta_{d}\left(Y\right)\right)
\]
is the $\left(d-1\right)$-truncation map.
\end{thm}

\begin{proof}
Follows from \lemref{h_d_mapping_spaces_operads}, \lemref{theta_d_universal_property_operad} (the universal
property of $\theta_{d}$) and \corref{char_of_ess_d-operad} analogously to the proof for $d$-categories.
\end{proof}

We conclude with a simple consequence of the theory of $d$-operads, that showcases the effectiveness of the strict model.
\begin{prop}
\label{prop:Alg_d_Category} Let $\mathcal{O}$ be an $\infty$-operad
and let $\mathcal{U}$ be an (essentially) $d$-operad. The $\infty$-category
$\underline{\alg}_{\mathcal{O}}\left(\mathcal{U}\right)$ is an (essentially)
$d$-category.
\end{prop}

\begin{proof}
Since an $\infty$-operad $\mathcal{U}$ is an essentially $d$-operad if and only if it is equivalent to a (strict) $d$-operad, it is enough to prove the strict version. 
By definition, the $\infty$-category $\alg_{\mathcal{O}}\left(\mathcal{U}\right)$ is a full subcategory of $\fun\left(\mathcal{O}^{\otimes},\mathcal{U}^{\otimes}\right)$.
For $d\ge1$, the $\infty$-category $\mathcal{U}^{\otimes}$ is a
$d$-category and, therefore, by T.2.3.4.8, the $\infty$-category $\fun\left(\mathcal{O}^{\otimes},\mathcal{U}^{\otimes}\right)$
is a $d$-category as well. Hence, every full subcategory of it is a $d$-category. For $d=0$, by \lemref{theta_d_universal_property_operad}
we can assume that $\mathcal{O}^{\otimes}$ is a $0$-operad as well and
therefore both $\mathcal{O}^{\otimes}$ and $\mathcal{U}^{\otimes}$
are skeletal 1-categories with faithful projection to $\finpt$. Observing
that $\alg_{\mathcal{O}}\left(\mathcal{U}\right)$ is a full subcategory
of $\fun_{\finpt}\left(\mathcal{O}^{\otimes},\mathcal{U}^{\otimes}\right)$
and using the faithfulness of the projections to $\finpt$, we see
that the mapping spaces are either empty or singletons. For $d=-1$,
the claim is obvious.
\end{proof}

\bibliography{Bib_EHA}
\bibliographystyle{alpha}

\end{document}